\documentclass[12pt]{amsart}
\usepackage{amsmath,amscd, xcolor, setspace, amsfonts, amssymb, bbm, amsthm, tikz, tikz-cd, enumerate, graphicx,verbatim,todonotes,mathtools}
\usetikzlibrary{calc,decorations.markings}
\usepackage[margin=1in]{geometry}
\usepackage[shortalphabetic]{amsrefs}

\newtheorem{thm}{Theorem}[section]
\newtheorem{theorem}[thm]{Theorem}
\newtheorem{prop}[thm]{Proposition}

\newtheorem{lemma}[thm]{Lemma}
\newtheorem*{lemma*}{Lemma}

\newtheorem{cor}[thm]{Corollary}
\newtheorem{remark}[thm]{Remark}

\newcommand{\Q}{\mathbb{Q}}

\newcommand{\N}{\mathbb{N}}
\newcommand{\F}{\mathbb{F}}
\newcommand{\C}{\mathbb{C}}

\newcommand{\res}{\underset{s=1}{\text{Res }}}

\numberwithin{equation}{section}

\allowdisplaybreaks

\title{Enumerating $D_4$ quartics and a galois group bias over function fields} 
\date{\today}
\author{Daniel Keliher}
\address{Department of Mathematics, Tufts University, 503 Boston Ave,
	Medford, MA 02144}
\email{daniel.keliher@tufts.edu}

\begin{document}
\maketitle

\begin{abstract}
We give an asymptotic formula for the number of $D_4$ quartic extensions of a function field with discriminant equal to some bound, essentially reproducing the analogous result over number fields due Cohen, Diaz y Diaz, and Olivier, but with a stronger error term. We also study the relative density of $D_4$ and $S_4$ quartic extensions of a function field and show that with mild conditions, the number of $D_4$ quartic extensions can far exceed the number of $S_4$ quartic extensions. 
\end{abstract}

\section{Introduction}

If $F$ is a number field, the number of $D_4$ and $S_4$ quartic extensions of $F$ with bounded discriminant is understood by work of Cohen, Diaz y Diaz, and Olivier \cite{CDO}, and of Bhargava, Shankar and Wang \cite{BSW}, respectively. In recent work, Friedrichsen and the author \cite{FK} study the relative sizes of these quantities and prove that $100\%$ of quadratic number fields have arbitrarily many more $D_4$ quartic than $S_4$ quartic extensions. 

In this note we seek to recover the results of \cite{FK} but over function fields. The result of Bhargava, Shankar, and Wang counting $S_4$ quartic extensions still applies in the function field setting. For counting $D_4$ extensions, the work of Cohen, Diaz y Diaz, and Olivier, while expected to generalize, has hitherto been stated only for number fields.

As such, our first task is to enumerate $D_4$ quartic extensions of function fields. Throughout, let $F$ be a function with constants $\mathbb{F}_q$ of characteristic not 2, and let $N_d^F(G; q^{2n})$ be the number of degree $d$ extensions of $F$ with Galois closure $G$ over $F$ and discriminant equal to $q^{2n}$.\footnote{If $F$ is a function field of genus $g$, its absolute discriminant is $D_F=q^{2g-2}$.} 

\begin{theorem}\label{D4counts}
Let $F$ be a function field with constants $\F_q$ of characteristic not 2 and $q \geq 5$. Then, 

\begin{equation}\label{D4error12}
N_4^F(D_4; q^{2n}) =  q^{2n}\log q  \sum_{\substack{K \\ [K:F]=2}} \frac{\res \zeta_K(s)}{D_K^2 \zeta_F(2)} 
 + O(n 4^n q^{n+1})
\end{equation}
where $D_K$ is the absolute discriminant of $K$, and $\zeta_K(s)$ is the Dedekind zeta function of $K$.
\end{theorem}

The main tool in proving Theorem \ref{D4counts} is an effective count of quadratic extensions of function fields. 

\begin{theorem}\label{QuadCount}
Let $F$ be a function field of genus $g_F$ with constants $\F_q$. Then, 
\begin{equation}\label{error14}
N_2^F(S_2; q^{2n}) = \frac{q^{2n} \log q}{\zeta_F(2)} ~ \res \zeta_F(s) + O\left(A^{g_F}q^{\frac n 2 + \frac{2g_F+1}{4}}\right)
\end{equation}
where $A$ and the implied constants are absolute.
\end{theorem}

For number fields, the expected error term for $N_2^F(S_2, X)$ is $O(X^{1/4 + \varepsilon})$, essentially mirroring the conjectured error term to obtain when enumerating the number of square-free integers up to some bound \cite{Pappalardi}, though this is not known even assuming the Riemann Hypothesis. As $q \to \infty$, we are seeing this error term reflected on the funcion field side in Theorem \ref{QuadCount}. Likewise, Theorem \ref{D4counts} has an error term analagous to $O(X^{1/2 + \epsilon})$, which one might  expect to be the best possible for enumerating $D_4$ quartic extensions over a number field, though again this is not known.

Theorem \ref{D4counts} together with the $n=4$ case of \cite[Theorem 1.b]{BSW} together imply that $N_4^F(D_4; q^{2n})$ and $N_4^F(S_4; q^{2n})$ have the same order of magnitude. It's natural then to compare their relative sizes. 

\begin{theorem}\label{MKff}
For any genus $g$ function field $F$ with constants $\mathbb{F}_q$,
\begin{equation}\label{ratio}
    \lim_{n \to \infty}\frac{N_4^F(D_4;q^{2n})}{N_4^F(S_4;q^{2n})} \gg \#Cl_F[2] \left(1-\frac{1}{\sqrt{q}}\right)^{4g-2}
\end{equation}
\end{theorem} 

By specializing to hyperelliptic extensions $F$ of $\mathbb{F}_q(t)$, we can make an analagous statement for a positive proportion of all such $F$.

\begin{theorem} \label{100per}
For any $g \geq 2$ and $q$ a power of an odd prime, the proportion of  genus $g$ hyperelliptic extensions $F$ of the rational function field $\mathbb{F}_q(t)$ such that
\[
\lim_{n \to \infty}\frac{N_4^F(D_4;q^{2n})}{N_4^F(S_4;q^{2n})} \gg g^{\frac{1}{2}\log 2}\left(1-\frac{1}{\sqrt{q}}\right)^{4g-2}
\]
is at least $1 - O\left(\frac{1}{\log g}\right)$.
\end{theorem}

Observe that as $q \to \infty$ the lower bound from Theorem \ref{100per} aproaches $g^{\frac{1}{2}\log 2}$. Thus, Theorem \ref{100per} implies, as $q \to \infty$ and for $g$ sufficiently large, an arbitrarily large proportion of genus $g$ hyperelliptic extensions of $\mathbb{F}_q(t)$ have arbitrarily many more $D_4$ quartic than $S_4$ quartic extensions. 

Later, in Theorem \ref{basechange}, we extend Theorem \ref{MKff} in a different direction than Theorem \ref{100per} by giving conditions under which one can artificially inflate the number $D_4$ relative to the number of $S_4$ by base changing the function field under consideration to one with a larger field of constants.  

Throughout, function fields will be taken to have constants $\mathbb{F}_q$ where 2 does not divide $q$. In our setting, these function fields correspond to smooth projective and geometrically connected curves over $\mathbb{F}_q$.
 
 In Section 2 we prove prove Theorem \ref{QuadCount}.  In Section 3 we collect more results from field counting and prove Theorem \ref{D4counts}. The main idea is to count quadratic-on-quadratic extensions of a ground field using the estimates obtained in Section 2. In Section 4 we turn to the study of the ratio $N_4^F(D_4; X)/N_4^F(S_4; X)$ and prove Theorem \ref{MKff}. Finally, in Section 5, we study the statistics of the number of irreducible factors of polynomials over $\F_q$ in order to prove Theorem \ref{100per}.
 
\section{Enumerating Quadratic Extensions}
Throughout, let $F$ be the function field of a smooth projective and geometrically connected curve of genus $g_F$ over $\mathbb{F}_q$. We denote by $\text{Cl}(F)$ the class group of $F$, and by $\text{Cl}_F[2]$ those elements of the class group with order dividing 2. The goal of this section is to prove an estimate for the number of quadratic extensions of $F$ with discriminant equal to $q^{2n}$. 

Our first goal is to prove Theorem \ref{QuadCount}. To begin, we'll focus on the Dirichlet series

\begin{equation}
    \phi_{F,2}(s) \coloneqq \sum_{\substack{K \\ [K:F]=2}} \frac{1}{D_{K/F}^s}
\end{equation}
where the sum ranges over quadratic extensions $K$ over $F$; $D_{K/F}$ is the norm of the relative discriminant ideal. Its coefficients will determine the quantity $N_2^F(S_2, q^{2n})$ of interest. We first obtain the following characterization of $\phi_{F,2}(s)$:

\begin{lemma}\label{QuadSeries}
If $\phi_{F,2}(s)$ is as above, then, 
$$\phi_{F,2}(s)=\frac{2}{\zeta_F(2s)} \sum_{\chi \in  \widehat{\mathrm{Cl(F)}[2]}} L(s,\chi)$$
where the sum ranges over the group characters of $\mathrm{Cl}(F)[2]$.
\end{lemma}

\begin{proof}
Let $\mathcal{O}_F$ be the integral closure of $\mathbb{F}_q[x]$ in $F$. Then we can write any quadratic extension $K/F$ as $F(\sqrt{\alpha})$ for some non-square $\alpha \in \mathcal{O}_F$. Write $\alpha\mathcal{O}_F=\mathfrak a \frak{b}^2$ with $\frak a$ square free, this determines the discriminant, $D_{K/F}=|\frak a|$. Further, if for some $\alpha' \in F$ we had $F(\sqrt{\alpha}) \simeq F(\sqrt{\alpha'})$ and $\alpha'\mathcal{O}_F=\frak a \frak{b}'^2$ then $\frak b$ and $\frak{b}'$ belong to the same ideal class in $\text{Cl}(F)$.

Indeed, we may think of any quadratic extension $K/F$ is determined by a choice of $\frak a$ (giving the discriminant), a choice of class $[\frak b] \in \text{Cl}(F)$ and a unit $u$ (cf. \cite[Lemma 3.3]{CDO}). 

Note that if $u\neq u'$ are two non-square units, then $F(\sqrt{\alpha}) \ncong F(\sqrt{u\alpha})$, but $F(\sqrt{u\alpha}) \cong F(\sqrt{u'\alpha})$ since $u$ and $u'$ differ by a square as the unit group of a finite field is cyclic. This establishes a bijection between quadratic extensions $K$ over $F$ and triples of the form $(\frak a, [\frak b], u)$ where $\frak a$ is a square-free ideal of $\mathcal{O}_F$, $[\frak b]$ is a class of ideal in $\text{Cl}(F)$ such that $\frak a \frak b^2 = \alpha\mathcal{O}_F$ and $u \in \mathcal{O}_F^{\times}/\mathcal{O}_F^{\times 2}$.
Now, by orthogonality, the number of quadratic extensions $K$ over $F$
with discriminant $\mathfrak{a}$ is 
$$2\sum_{\chi \in \widehat{\mathrm{Cl}(F)[2]}} \chi(\mathfrak{a}).$$

Observing that
 $$N_2^F(S_2; q^n) = 2\sum_{\substack{\frak a \\ |\frak a|=q^n}} \sum_{\chi \in  \widehat{\text{Cl(F)}[2]}} \chi(\frak a),$$ 
 where the sum is over square-free ideals $\frak a$ of $\mathcal{O}_F$, we then have
\begin{align}
\phi_{F,2}(s) &= \sum_{n=1}^\infty   \sum_{\substack{K \\ [K:F]=2 \\ D_{K/F}=q^{n}}} q^{-ns} = \sum_{n=1}^\infty N_2^F(S_2, q^n)q^{-ns} \label{n2fseries} \\
\notag &= 2 \sum_{n=1}^\infty \sum_{\substack{\frak a \\ |\frak a|=q^n}} \sum_{\chi \in  \widehat{\text{Cl(F)}[2]}} \chi(\frak a) |\frak a|^{-s} 
=2  \sum_{\substack{\frak a}} \sum_{\chi \in  \widehat{\text{Cl(F)}[2]}} \chi(\frak a) |\frak a|^{-s} \\
\notag &= 2 \sum_{\chi \in  \widehat{\text{Cl(F)}[2]}} \frac{L(s,\chi)}{L(2s, \chi^2)} 
= \frac{2}{\zeta_F(2s)} \sum_{\chi \in  \widehat{\text{Cl(F)}[2]}} L(s,\chi)
\end{align}
as needed.
\end{proof}
This is reminiscent of the expression for the Dirichlet series obtained in \cite[Theorem 1.3]{CDO}, but the computations are significantly simpler in the function field setting. In particular, 2 does not divide the characteristic of $F$.

Before proceeding to the proof of Theorem \ref{QuadCount}, we state some additional supporting lemmas needed for the proof and in further sections.  A detailed discussion of the following two lemmas can be found, for example, in Rosen's book \cite{Rosen}.  

Making the change of variables $u=q^{-s}$, the zeta function of $F$ is given by
\begin{equation} \label{ffzeta}
\zeta_F(s) := Z_F(u) = \frac{L_F(u)}{(1-u)(1-qu)} \text{  with } L_F(u) \in \mathbb{Z}[u].
\end{equation}
Over $\mathbb{C}$, $L_F(u)$ factors as 
\begin{equation}\label{WC}
    L_F(u)=\prod_{i=1}^{2g_F} (1-\pi_iu)
\end{equation} where $g_F$ is the genus of the function field $F$.  We mainly will use the fact that the Riemann Hypothesis for Function Fields implies that the inverse roots of $L_F(u)$ have absolute value $\sqrt q$. We will make frequent use of the fact that $\zeta_F(s)$ and related $L$-functions may be written as rational functions. 

We also need an estimate for the 2-part of the class group of $K$, $\#\text{Cl}_F[2]$. See \cite[Prop. 5.11]{Rosen}.
\begin{lemma}\label{h2est}
With the notation all as before, $\#\mathrm{Cl}_F[2] \leq \#\mathrm{Cl}_F \leq (1+\sqrt q)^{2g_F}$.
\end{lemma}

We are now ready to prove Theorem \ref{QuadCount}.
\begin{proof}[Proof of Theorem \ref{QuadCount}]
Beginning with (\ref{n2fseries}) and making the convinient change of variables $u \coloneqq q^{-s}$ we have 
$$\phi_{F,2}(s) = \sum_{n=1}^\infty N_2^F(S_2, q^n)u^{s}.$$
Write $\mathcal{L}_\chi(u)$ for the $L$-function $L(s, \chi)$ after the change of varibles to $u = q^{-s}$.  By applying Cauchy's integral formula to the expression for $\phi_{F,2}(s)$ of Lemma \ref{QuadSeries}, we have 

\begin{equation}\label{coefint}
N_2^F(S_2; q^{2n}) = \frac{1}{ \pi i} \sum_{\chi \in  \widehat{\text{Cl(F)}[2]}}  \oint_\gamma \frac{\mathcal{L}_\chi(u)}{Z_F(u^2)u^{2n+1}}du
\end{equation}
where $\gamma$ is a circle of sufficiently small radius $\varepsilon > 0$ centered at $u=0$. We compute $a_n$ by expanding the radius of $\gamma$ and computing residues of the integrand. The integrand has poles at the poles of $\mathcal{L}_\chi(u)$ and at the zeros of denominator. 

Observe that if $\chi$ in the sum of \eqref{coefint} is the trivial character, then $\mathcal{L}_\chi(u) = Z_F(u)$ and we get a pole at $u=\frac 1 q$. In other cases, $\mathcal{L}_\chi(u)$ doesn't contribute a pole, and so we will focus on the trivial $\chi$ case. The rest will follow in a similar fashion.  Note, the zeros of $Z_F(u^2)$ occur, by the Weil Conjectures,  only for values of $u$ where $|u|= q^{-\frac 1 4}$. 

Let $\gamma'$ be a circle centered at $u=0$ with radius $R$ satisfying $q^{-1} < R < q^{-\frac 1 4}$. 
In shifting the contour from $\gamma$ to $\gamma'$, this constraint forces us to pick up the residue of the integrand at $u= \frac 1 q$ but not any of the residues contributed from the zeros of $Z_F(u^2)$. For any such $R$ we have 
\begin{equation}\label{coefintres}
\frac{1}{ \pi i} \oint_{\gamma} \frac{Z_F(u)}{Z_F(u^2)u^{2n+1}}du = \frac{1}{ \pi i} \oint_{\gamma'} \frac{Z_F(u)}{Z_F(u^2)u^{2n+1}}du - \underset{u = \frac 1 q}{2\text{Res}} \frac{Z_F(u)}{Z_F(u^2)u^{2n+1}}.
\end{equation}
One verifies, using the change of variables $u=q^{-s}$, that
$$
\underset{u = \frac 1 q}{\text{Res}} \frac{Z_F(u)}{Z_F(u^2)u^{2n+1}} = -\frac{q^{2n} \log q}{\zeta_F(2)} \res \zeta_F(s).
$$
This shows the right side of \eqref{coefintres} yields 
\begin{equation}\label{Rcases}
N_2^F(S_2; q^{2n}) = \frac{2q^{2n} \log q}{\zeta_F(s)} \res \zeta_F(s) + \frac{1}{ \pi i}\oint_{\gamma'}\frac{Z_F(u)}{Z_F(u^2)u^{2n+1}} du.
\end{equation}
Using the factorization of $Z_F(u)$ as a rational function \cite[Theorem 5.6]{Rosen}, bound the integral above by bounding the integrand. Set 
\begin{equation}\label{errorR}
E := \left| \frac{Z_F(u)}{Z_F(u^2)u^{2n+1}} \right| = \left|\frac{(1-u^2)(1-qu^2)\prod_{i=1}^{2g_F}(1-\alpha_iu)}{(1-u)(1-qu)\prod_{i=1}^{2g_F}(1-\alpha_iu^2)u^{2n+1}}\right|
\end{equation}
where $|\alpha_i|= \sqrt{q}$. 

To complete our understanding of $\eqref{Rcases}$, we will bound $E$ in the cases $R=q^{-\frac 1 2}$ and $R=q^{-\frac 1 4 - \varepsilon}$ for some small $\varepsilon > 0$. In both cases  we bound the size of $Z_F(u)/Z_F(u^2)$ when $|u|=R$.  In bounding the numerator from above and denominator from below, we simply suppose each term is as large or small as possible, i.e. bounds following from taking $\alpha_i = \pm \sqrt{q}$ and $u= \pm R$.

First, setting $R = q^{-\frac 1 2}$, one finds 
\begin{equation}\label{Rq1/2}
E \leq 2 \left( \frac{2}{1-\frac{1}{\sqrt{q}}}\right)^{2g_F}\left( \frac{1+\frac{1}{\sqrt q}}{\sqrt q -1}\right)q^{n}.
\end{equation} 
Second, setting $R=q^{-\frac 1 4 - \varepsilon}$, one finds,  
\begin{equation}\label{Rq1/4eps}
E \leq \left( \frac{(1+q^{\frac 1 4 - \varepsilon})^{2g_F}(1+q^{- \frac 1 4 - \varepsilon})(1 + q ^{\frac 3 4 - \varepsilon})}{(1-q^{-\varepsilon})^{2g_F}(q^{\frac 3 4 - \varepsilon }-1)}\right)q^{\frac{2n - 1 + \varepsilon}{4}}.
\end{equation} 

Repeating these computations for nontrivial $\chi$ which contributed to the remaining terms in \eqref{coefint} suffices to prove the lemma. 
We then multiply each $E$ by $2\pi R$, the length of the circle of integration, to get an error term. After doing so and substituting $2n$ for $n$,  (\ref{Rq1/2}) and (\ref{Rq1/4eps}) yield, respectively, 
\begin{equation}\label{noeps}
N_2^F(S_2; q^{2n}) = \frac{2q^{2n} \log q}{\zeta_F(2)} ~ \res \zeta_F(s) + O\left( 2 \left( \frac{2}{1-\frac{1}{\sqrt{q}}}\right)^{2g_F}\left(1+\frac{1}{\sqrt q}\right)q^{n}\right)
\end{equation}
and, for any $\varepsilon > 0$,
\begin{equation}\label{witheps}
N_2^F(S_2; q^{2n}) = \frac{2q^{2n} \log q}{\zeta_F(2)} ~ \res \zeta_F(s) + O\left(\left( \frac{(1+q^{\frac 1 4 - \varepsilon})^{2g_F}(1+q^{- \frac 1 2 - \varepsilon})(1 + q ^{\frac 3 4 - \varepsilon})}{(1-q^{-\varepsilon})^{2g_F}}\right)q^{\frac{n +\varepsilon}{2}}\right).
\end{equation}
Note that each of the two formulas in \eqref{noeps} and \eqref{witheps} may have utility in their own right. We now make a choice of $\varepsilon > 0$ in (\ref{witheps}) that gives the statment of the theorem and a formula which is amenable to computation. 

Set $\varepsilon = 1/\log q$. Then the error term in (\ref{witheps}) becomes $$O\left(A^{g_F}q^{\frac n 2 + \frac{2g_F+1}{4}}\right)$$
where $A = (1-1/e)^{-2}$. The implied constant can be computed from (\ref{Rq1/4eps}) and is absolute.
\end{proof}
\begin{remark}
After fixing $q$ (and possibly some $\varepsilon > 0)$ in the proof above, the computations for $E$ for all $\chi$ in the proof of Theorem \ref{QuadCount} can be converted into explicit bounds on the error in the formulae of Theorem \ref{QuadCount}.
\end{remark}
We conclude the section with a uniform upper bound on the number of quadratic extensions of $F$. Though it isn't a strong as the preceding theorem, it will simplify some computations later.

\begin{lemma}\label{quadbound}
Fix a function field $F$ with constants $\mathbb{F}_q$. For any $n \geq 0$, 
$$N_2^F(S_2; q^{2n}) \ll \#\mathrm{Cl}_F[2] B^{2g_F}  q^{2n+1}$$
where $B= \left(\frac{1+e^{-1}q^{-1/2}}{1-e^{-2}q^{-3/2}}\right)^{2}$.
\end{lemma}

\begin{proof}
We will mimic the method used in the proof of Theorem \ref{QuadCount} and use the same notation. Starting with \eqref{coefint}, bound the integrand but now with $R = 1/eq$. This has the effect of avoiding the pole at $u=1/q $ contributed by the integrand when $\chi$ is the trivial character. 

When $\chi$ is the trivial character and $R=1/eq$,  
$$E \ll B^{g_F}q^{2n+1}$$

where 
$$B= \left(\frac{1+e^{-1}q^{-1/2}}{1-e^{-2}q^{-3/2}}\right)^{2}.$$

 The same bound, albeit with different implied constants, is obtained for the remaining $\#\mathrm{Cl}_F[2]-1$ nontrival characters, $\chi$, appearing in \eqref{coefint}. Whence, 
$$N_2^F(S_2; q^{2n}) \ll \#\mathrm{Cl}_F[2] B^{2g_F} q^{2n+1}.$$

\end{proof}

\section{Enumerating $D_4$ Extensions}
We are now ready to enumerate $D_4$ quartic extensions of a function field $F$, i.e. to prove Theorem \ref{D4counts}. We'll first state some lemmas which, when taken with the results of Section 2, will suffice to prove the theorem.

\begin{lemma}[\cite{CDO}, Corollary 2.3]  \label{CDO}
Fix a global field $F$. We have the formal equality: 
\begin{equation} \label{CDOcount}
\sum_{\substack{K \\ [K:F]=2}}\sum_{\substack{L \\ [L:K]=2 \\ D_{L/F} \leq X}} 1 = 2N_4^F(D_4; X) + N_4^F(C_4; X) + 3N_4^F(V_4; X)\end{equation}
\end{lemma}

The proof of this fact is the same as in \cite{CDO} and relies only on the Galois correspondence and the structure of quartic extensions obtained as quadratic-on-quadratic extensions of $F$. Since we can only have discriminants which are even powers of $q$, we'll take $X=q^{2n}$. 

The idea is to use \eqref{CDOcount} to understand $N_4^F(D_4;X)$. Now we state some lemmas that will control the last two terms of \eqref{CDOcount}. Then the remainder of the section will be devoted to understanding the lefthand side of \eqref{CDOcount}. 

\begin{lemma} \label{wright}
If $F$ is a global field, then as $n \to \infty$, $$N_4^F(C_4; q^{2n}) = O\left(q^n\right)$$
and $$N_4^F(V_4; q^{2n}) = O\left(q^n\log(q^{2n})^2 \right).$$
\end{lemma}
\begin{proof}
Both of these estimates follow from \cite[Theorem 1]{W} and applications of Tauberian theorems. 
\end{proof}

\begin{remark}\label{D4rawform}
An immediate consequence of Lemmas \ref{CDO} and \ref{wright} is that 

\begin{equation} \label{D4SumandError}
N_4^F(D_4; q^{2n}) = \frac 1 2 \sum_{\substack{K \\ [K:F]=2}} \sum_{\substack{L \\ [L:K]=2 \\ D_L = q^{2n}}} 1 + O\left((q^n \log(q^{2n})^2\right).
\end{equation}
Indeed, analyzing the sum above will constitute the main idea of the the proof. 
\end{remark}

We are now ready to prove Theorem \ref{D4counts}. 

\begin{proof}[Proof of Theorem \ref{D4counts}]
Let $D_F$ be the discriminant of $F$ and let $D_{L/F}$ be the norm of the relative discriminant ideal of $L/F$. We have the relation that $D_FD_{L/F}^2 = D_L$ \cite[Theorem A]{JP}. 

We are trying to compute the sum (\ref{D4SumandError}) in Remark \ref{D4rawform}.

\begin{equation} \label{mainsum}
N_4^F(D_4; q^{2n}) = \frac 1 2 \sum_{\substack{K \\ [K:F]=2}} \sum_{\substack{L \\ [L:K]=2 \\ D_L = q^{2n}}} 1 + O\left((q^n \log(q^{2n})^2\right)
\end{equation}
where the sum counts quadratic-on-quadratic extensions of $F$ of discrimiant $q^{2n}$ and the error term comes from those biquadratic extensions with galois group $C_4$ or $V_4$.

To control the sum above, we introduce an auxiliary parameter $j$, which will start at $-1$ and then run over integers up to $n/2$, and rewrite \eqref{mainsum} as:
\begin{equation}
\sum_{j \leq n/2} \sum_{\substack{K \\ [K:F]=2 \\ D_K=q^{2j}}}\sum_{\substack{L \\ [L:K]=2 \\ D_{L/K}=q^{2n-4j}}} 1 = \sum_{j \leq n/2}\sum_{\substack{K \\ [K:F]=2 \\ D_K=q^{2j}}} N^K_2(q^{2n-4j}).
\end{equation}
Note that the parameter $j$ controls the discriminant (and also the genus) of the intermediate field $K$. In particular, the genus of each intermediate field $K$ is $g_K=j+1$.

First apply the first estimate for $N^K_2(S_2; q^{2n-4j})$ from Theorem \ref{QuadCount}:

\begin{align}\label{sub12}
\sum_{j \leq n/2 }\sum_{\substack{K \\ [K:F]=2 \\ D_K=q^{2j}}} & N^K_2(S_2; q^{2n-4j}) \\
\notag &= \sum_{j \leq n/2 }\sum_{\substack{K \\ [K:F]=2 \\ D_K=q^{2j}}} \left[\frac{2q^{2n-4j} \log q}{\zeta_F(2)} ~ \res \zeta_K(s) + O\left( 2 \left( \frac{2}{1-\frac{1}{\sqrt{q}}}\right)^{2g_K}\left(1+\frac{1}{\sqrt q}\right)q^{n-2j}\right)\right]
\end{align}

The main term of the above is 
\begin{equation}
 2q^{2n}\log q \sum_{j \leq n/2 }\sum_{\substack{K \\ [K:F]=2 \\ D_K=q^{2j}}} \frac{\res \zeta_K(s)}{q^{4j} \zeta_F(2)}.
\end{equation}

We can compare the sum above to the same, but untruncated, series where $j$ runs over all the integers, to find:
\begin{equation}\label{infversion}
 2q^{2n}\log q \sum_{j \leq n/2 }\sum_{\substack{K \\ [K:F]=2 \\ D_K=q^{2j}}} \frac{\res \zeta_K(s)}{q^{4j} \zeta_K(2)} =
 2q^{2n}\log q \sum_{j=0}^\infty \sum_{\substack{K \\ [K:F]=2 \\ D_K=q^{2j}}} \frac{\res \zeta_K(s)}{q^{4j} \zeta_K(2)} + O(q^{-n}).
\end{equation}
The above can be seen by estimating the infinite series with a geometric series using upper and lower bounds on $\res \zeta_K(s)$ and $\zeta_K(2)$ that are independent of $K$. We'll see the error term of $\label{infversion}$ is subsumed by the error term from \eqref{sub12}.

Now we compute the error term in \eqref{sub12}. For ease of notation, set $c_q := \frac{2}{1-q^{-\frac 1 2}}$.

Using Lemma \ref{quadbound}, the error term can be seen to be of size 
\begin{align*}
q^n \sum_{j \leq n/2 }\sum_{\substack{K \\ [K:F]=2 \\ D_K=q^{2j}}} 2c_q^{2j+2}\left(1+\frac{1}{\sqrt q}\right) q^{-2j} &=  q^n \sum_{j \leq n/2 }N_2^F(S_2, q^{2j}) 2c_q^{2j+2}\left(1+\frac{1}{\sqrt q}\right)q^{-2j} \\
&\ll q^{n+1}c_q^2 h_2(F) B^{g_F} \left(1+\frac{1}{\sqrt q}\right)\sum_{j \leq n/2}  c_q^{2j} \\
&\ll q^{n+1}h_2(F)B^{g_F}nc_q^n.
\end{align*}
This gives an error term $O\left(nc_q^n q^{n+1} \right)$ which dominates the error term in \eqref{mainsum}. Finally, notice $c_q < 5$ for any choice of $q$, and $c_q < 4$ when $q\geq 5$. 
We have thus proved (\ref{D4error12}) of Theorem \ref{D4counts}.
\end{proof}
 
 \begin{remark}
 Using \eqref{error14} of Theorem \ref{QuadCount} in the proof above does not yield a better error term.
 \end{remark}
 This essentially recovers the asymtotic given by Cohen, Diaz y Diaz, and Olivier in \cite{CDO} in the function field setting. The easier analysis granted to us by the Weil Conjectures lets us improve on the number field version, yielding more than just an asymptotic.  If we take $n \to \infty$ in Theorem \ref{D4counts}, this connection is made visibly clear by the following formulae.
\begin{cor} With the same notation as above, 
$$\lim_{n \to \infty} \frac{N_4^F(D_4; q^{2n})}{q^{2n}} = \frac{\log q}{2} \sum_{j=-1}^
\infty \sum_{\substack{K \\ [K:F]=2 \\ D_K=q^{2j}}} \frac{\res \zeta_K(s)}{q^{4j} \zeta_F(2)}= \frac{\log q}{2} \sum_{\substack{K \\ [K:F]=2}} \frac{\res \zeta_K(s)}{D_K^2 \zeta_F(2)}  .$$
\end{cor}

\section{Enumerating the $D_4$-$S_4$ disparity.}
We'll turn now to considering the ratio  $N_4^F(D_4; q^{2n})/N_4^F(S_4; q^{2n})$. In \cite{FK}, this quantity is studied in the case that $F$ is a number field. Of interest is the case when $F$ has more $D_4$ than $S_4$ quartic extensions. One of the main results of \cite{FK} is that this ratio may be skewed arbitrarily much in favor of the $D_4$ quantity.  To obtain the analagous result for function fields, i.e. Theorem \ref{100per}, we'll first address ourselves to the proof of Theorem \ref{MKff} to give a lower bound for  $N_4^F(D_4; q^{2n})/N_4^F(S_4; q^{2n})$.

\begin{proof}[Proof of Theorem \ref{MKff}]
First, we underestimate $N_4^F(q^{2n};D_4)$ by restricting the sum in Theorem \ref{D4counts} to be over only those $K$ which are unramified over $F$. Given such an extension $K$, the different divisor $\text{Diff}(L/K) = \sum_{\frak P} c_{\frak P} \frak P$ is 0, i.e. $c_{\frak P}=0$ for all primes $\frak P$ of $K$. So Riemann-Hurwitz tells us $g_K=2g_F-1$. Note also that class field theory tells us that there are $\#\text{Cl}_F[2]-1$ such extensions $K$.

Further, in restricting ourselves to unramified extension, the $q^{4j}$ term appearing in Theorem \ref{D4counts} can be rewrriten as $q^{4(2g_F-2)}$ via $j=g_k-1 = 2g_F-2$. Our underestimate for $N_4^F(D_4; q^{2n})$ is 

\begin{align}\label{D4low}
   \lim_{n \to \infty}\frac{N_4^F(D_4;q^{2n})}{q^{2n}} &\gg \log q \sum_{\substack{K \\ [K:F]=2 \\ \text{unram.}}}  \frac{\res\zeta_K(s)}{q^{4(2g_F-2)}\zeta_F(2)} \\
    &\notag= \log q \sum_{\substack{K \\ [K:F]=2 \\ \text{unram.}}} \frac{L(1, \chi_{K/F})\res \zeta_F(s)}{q^{4(2g_F-2)}\zeta_F(2)} \\
    &\notag \geq \log q ~ \#\text{Cl}_F[2]\left(1-\frac{1}{\sqrt{q}}\right)^{2(2g_F-1)} \frac{\res \zeta_F(s)}{q^{4(2g_F-2)}\zeta_F(2)}.
\end{align}

Note that this is not such a terrible truncation of the sum in Theorem \ref{D4counts}: Estimate each residue of $\zeta_K(s)$ at $1$, then factoring these estimates outside the sum leaves a rapidly convergent series, to which the main contributions are made by small $j$ terms, including the unramified extensions. 

We'll now estimate $N_4^F(S_4;q^{2n})$. The degree 4 case of \cite[Theorem 1.b]{BSW} gives the number of $S_4$ quartic extensions of $F$ with relative descriminant equal to some power of $q$. Using the relation $D_{L/F} D_{F}^4 = D_L$ \cite[Theorem A]{JP}, we can find the number of $S_4$ quartic extensions $L$ over $F$ with absolute discriminant $q^{2n}$ by counting the number of $S_4$ quartic extensions $L$ over $F$ with relative discriminant of norm $D_{L/F}=q^{2n}/D_F^4 = q^{2n}/q^{4(g_F-2)}$. Whence, 
\begin{equation}\label{S4up}
\lim_{n \to \infty}\frac{N_4^F(S_4;q^{2n})}{q^{2n}}  \ll_F q^{-4(2g_F-2)} \log q ~\res \zeta_F(s). 
\end{equation}

Taking the ratio of \eqref{D4low} and \eqref{S4up} gives us the desired result. 
\end{proof}

The contribution from $\text{Cl}_F[2]$ in Theorem \ref{MKff} is not necessarily easily computed. We can base change our field $F$ over $\mathbb{F}_q$ to have some larger field of constants $\mathbb{F}_{q^m}$ to get a more explicit bound on the right of \eqref{ratio} of Theorem \ref{MKff}. In doing so, we can understand fully the contribution of $\text{Cl}_F[2]$.  

\begin{theorem}\label{basechange}
Let $F$ be the function field of a curve $C/\mathbb{F}_q$ of genus $g$. There exists a constant $m$ such that if we base change $C$ to be over $\mathbb{F}_{q^m}$ and let $F'$ be the corresponding function field, then
$$\lim_{n \to \infty}\frac{N_4^{F'}(q^{2n};D_4)}{N_4^{F'}(q^{2n};S_4)} \gg 2^{2g} \left(1-\frac{1}{\sqrt{q^m}}\right)^{4g-2}.$$
\end{theorem} 

\begin{proof}
Let $J_F$ be the Jacobian of $F$.  There is a natural map $J_F \hookrightarrow \text{Cl}_F$ and so also an injective map on the two-torsion: $J_F[2] \hookrightarrow Cl_F[2].$ In particular we can use $\#J_F[2]$ as a (possibly crude) proxy for $\#\text{Cl}_F[2]$. Consider the multiplication by 2 endomorphism, $[2]$, on the $\bar{\mathbb{F}}_q$ points of $J_F$,
$$J_F(\bar{\mathbb{F}}_q)\overset{[2]}{\longrightarrow} J_F(\bar{\mathbb{F}}_q).$$
This is a surjective, degree $2^{2g}$ map. The two torsion of $J_F(\bar{\mathbb{F}}_q)$ is given by  $\ker([2])$. We're looking at $J_F$ over $\bar{\mathbb{F}}_q$ and we're concerned only with an extension of the field of constants to a larger finite field, but we have the relation $J_F(\bar{\mathbb{F}}_q)^{\text{Gal}(\bar{\mathbb{F}}_q/\mathbb{F}_q)}= J_F(\mathbb{F}_q)$. 

Notice that the $2^{2g}$ points in $J_F(\bar{\mathbb{F}}_q)[2]$ are partitioned into Galois orbits under the action of $\hat{\mathbb{Z}} \cong \text{Gal}(\bar{\mathbb{F}}_q/\mathbb{F}_q)$. So there is some positive integer $m$ such that all of $J_F(\bar{\mathbb{F}}_q)[2]$ is stable under the action of $m \hat{\mathbb{Z}} < \hat{\mathbb{Z}}$. These are exactly the points of $J_F(\mathbb{F}_{q^m})[2]$.

Let $F'$ be the base extension of $F$ having constants $\mathbb{F}_{q^m}$. This is a constant extension of function fields and so the genus of $F'$ is also $g$ \cite[Chapter 8]{Rosen}, thus passing from the lower bound given by Theorem \ref{MKff} to the lower bound under consideration only requires is to make the substitution of $2^{2g}$ for $\#\text{Cl}_{F'}[2]$ since we have just demonstrated that $2^{2g} \leq \#\text{Cl}_{F'}[2]$.

\end{proof}

Observe then that we need only pick a high enough genus curve with suitably large $q$ in order to skew the $D_4$-$S_4$ ratio arbitrarily high in favor of the number of $D_4$ extensions. 

\section{Typical Behavior of Quadratic Function Fields}

Throughout this section, let $P_n$ be the set of degree $n$ square-free polynomials with coefficients in $\mathbb{F}_q$ and let $\pi(d)$ be the number of irreducible degree $d$ polynomials over $\mathbb{F}_q$. Recall $\#P_n=q^n-q^{n-1}$. 

Indeed, for everything that follows, all polynomials will be taken to have coefficients in $\mathbb{F}_q$. In light of Theorem \ref{MKff}, our interest is in understanding hyperelliptic curves $y^2=f(x)$ over $\mathbb{F}_q$ which have many two torsion elements in the class groups of their corresponding function fields. Such elements correspond to factors of $f(x)$, and so we will settle for understanding what the typical number of irreducible factors is for a ``random'' $f(x)$. Many results of a similar flavor, particularly regarding mean and variance, due to Knopfmacher and Knopfmacher can be found in \cite{KK}.

In the following two propositions $T=T(n)$ will be a function of $n \in \N$ such that both $T(n) \rightarrow \infty$ and $n-T(n) \rightarrow \infty$ as $n \rightarrow \infty$  We will make a convenient choice of $T$ later. 

Our primary tool for understanding the typical number of irreducible factors is the following theorem due to Chebyshev, stated here in the context of a finite sets:
\begin{theorem}[Chebyshev's Inequality]\label{Cheby}
Let $X$ be a finite set and $f\colon X \rightarrow \C$. Then if $\phi$ has mean and variance given respectively by 
$$\mu := \frac{1}{\#X}\sum_{x \in X}\phi(x) \text{ and } \sigma^2 := \frac{1}{\#X}\sum_{x \in X}
(\phi(x)-\mu)^2,$$
then for any $k$,
$$\#\{x \in X \mid |\phi(x)-\mu|\geq k \sigma \} \leq \frac{\#X}{k^2}.$$
\end{theorem}
For our considerations the $X$ of Theorem \ref{Cheby} will be $P_n$, and the $\phi$ will be the function counting the number of irreducible divisors of all $f \in P_n$ with degree bounded by $T$. As such, let $\omega_T(f)$ be the the number of irreducible divisors of $f$ with degree bounded by $T$.

\begin{prop}[``Expected Number'' of Factors]\label{mean} The expected number of irreducible factors with degree bounded by $T$ of polynomials in $P_n$ is
$$\mu \coloneqq \frac{1}{\#P_n}\sum_{f \in P_n} \omega_T(f) = \sum_{d\leq T} \frac{\pi(d)}{q^d+1} + O\left( q^{2T-n+1}\right).$$
\end{prop}

\begin{proof}
The mean number of irreducible factors with degree bounded by $T$ of polynomials ranging over $P_n$ is given by 
\begin{equation}\label{meandef}
\mu := \frac{1}{\#P_n} \sum_{f \in P_n} \sum_{\substack{p \text{ irred.}\\ p | f \\ \text{deg }p \leq T}} 1 = \frac{1}{\#P_n} \sum_{d \leq T} \sum_{\substack{p \text{ irred.} \\ \text{deg }p = d}} \sum_{\substack{f \in P_n\\ p | f }} 1
\end{equation}
 
First, for a fixed $d$ and irreducible $p$ of degree $d$, we aim to understand the innermost sum of (\ref{meandef}), i.e. $\sum_{\substack{f \in P_n\\ p | f }} 1$. This sum is counting the number of polynomials $f \in P_n$ which can be written as $f=gp$ for some $g \in P_{n-d}$. Since $f$ is square-free, we have $p \nmid g$. Now set
$$N_1:= \#\{f \in P_n \mid f=pg, g \in P_{n-d}, p \nmid g\}.$$
Computing $N_1$ is the same as understanding the number of $g$ with the given condition. So similarly, set  
$$N_2 := \#\{g \in P_{n-d} \mid g=ph, h \in P_{n-2d}, p \nmid h\}$$ 
and observe
$$N_1 = \#P_{n-d} - N_2.$$
One can inductively continue this procedure, if 
$$N_s = \#\{t \in P_{n-(s-1)d}\mid t=pu, u \in P_{n-sd}, p \nmid u \}$$
then $N_s = \#P_{n-sd} - N_{s+1}.$ The definition of $N_s$ is only sensical for $s+1$ up to $\lfloor \frac{n}{d}\rfloor$. One then finds that 
\begin{equation} \label{N1}
N_1 = \left(\sum_{k=1}^{\lfloor \frac{n}{d}\rfloor -1}(-1)^{k+1} \#P_{n-kd}\right) \pm N_{\lfloor \frac{n}{d}\rfloor}.
\end{equation}
We can bound $N_{\lfloor \frac{n}{d}\rfloor}$ trivially by $P_{n-\lfloor\frac{n}{d}\rfloor d}$. We have  $n-\lfloor\frac{n}{d}\rfloor d \leq d$ so, at worst $N_{\lfloor \frac{n}{d}\rfloor} = O(q^d)$. When we compute the mean and divide through by $\#P_n$, this error will be negligible. 

Beginning from \eqref{N1}, we find

\begin{align*}
N_1 = \left(\sum_{k=1}^{\lfloor \frac{n}{d}\rfloor -1}(-1)^{k+1} \#P_{n-kd}\right) + O(q^d) &=  \left(\sum_{k=1}^{\lfloor \frac{n}{d}\rfloor -1}(-1)^{k+1} (q^{n-kd}-q^{n-kd-1})\right) + O(q^d)\\
&=  (q^n-q^{n-1})\left(\sum_{k=1}^{\lfloor \frac{n}{d}\rfloor -1}(-1)^{k+1} q^{-kd}\right) + O(q^d) \\
&=(q^n-q^{n-1})\left(\frac{1}{q^d+1} + O(q^{-n-1})\right) + O(q^d) \\
&=(q^n-q^{n-1})\frac{1}{q^d+1} + O(q^d) 
\end{align*}

Substituting this last expression into \eqref{meandef} yields, 
\begin{align*}
\mu  &= \frac{1}{\#P_n} \sum_{d \leq T} \sum_{\substack{p \text{ irred.} \\ \text{deg }p = d}} \sum_{\substack{f \in P_n\\ p | f }} 1 \\
&= \sum_{d \leq T} \sum_{\substack{p \text{ irred.} \\ \text{deg }p = d}} \left( \frac{1}{q^d+1} + O(q^{d-n+1})\right) \\
&= \sum_{d\leq T} \pi(d)\left( \frac{1}{q^d+1} + O(q^{d-n+1})\right).
\end{align*}
Finally, using the prime number theorem for irreducible polynomials over $\mathbb{F}_q$, the above becomes 
$$\sum_{d\leq T} \frac{\pi(d)}{q^d+1} + O\left( q^{2T-n+1}\right)$$
as desired.
\end{proof}

\begin{prop}[The ``variance'' in the number of irreducible factors]\label{variance} The variance in the number of irreducible factors with degree bounded by $T$ of polynomials in $P_n$ is
$$\sigma^2 \coloneqq \frac{1}{\#P_n}\sum_{f \in P_n} (\omega_T(f) - \mu)^2 = \sum_{d} \pi(d)\frac{1}{q^d+1}\left(1-\frac{1}{q^d+1}\right) + O\left(q^{2(T-n+1)}\right).$$
\end{prop}

\begin{proof}
For a random variable $X$, the variance is given by $\mathbb{E}[X^2] -\mathbb{E}[X]^2$. One easily verifies the following analogue of that identity: 
\begin{equation}\label{variance}
\sigma^2 = \frac{1}{\#P_n} \sum_{f \in P_n} \left( \sum_{\substack{p \text{ irred.} \\ p\mid f \\ \text{deg }p \leq T}} 1 \right)^2 - \left(\frac{1}{\#P_n} \sum_{f \in P_n}  \sum_{\substack{p \text{ irred.} \\ p\mid f \\ \text{deg }p \leq T}} 1 \right)^2. 
\end{equation}
We'll compute the first term in the difference in \eqref{variance}:
\begin{equation}\label{EsqX}
\frac{1}{\#P_n} \sum_{f \in P_n} \left( \sum_{\substack{p \text{ irred.} \\ p\mid f \\ \text{deg }p \leq T}} 1 \right)^2 = \frac{1}{\#P_n} \sum_{f \in P_n}  \sum_{\substack{p,p' \text{ irred.} \\ p,p'\mid f \\ \text{deg }p,~ \text{deg } p' \leq T}} 1=\frac{1}{\#P_n} \sum_{d,d' \leq T}  \sum_{\substack{p,p' \text{ irred.} \\ \text{deg }p =d \\ \text{deg }p' = d'}} \sum_{\substack{f \in P_n \\ p,p' \mid f}} 1.
\end{equation} 

Let $N_n^{p_1,...,p_r}$ be size of the set of all polynomials $f \in P_n$ such that $p_1 | f, \hdots , p_r | f$. Note, for fixed irreducible polynomials $p$ and $p'$ of degrees $d$ and $d'$, respectively.  Set $N_n^{p_1,...,p_r} =\#P_n^{p_1,...,p_r}$ , then we have $N_n^{p,p'} = \sum_{\substack{f \in P_n \\ p,p' \mid f}} 1$, which is the innermost sum on the right size of \eqref{EsqX}. We have then that 
$$N^{p,p'}_n = N^p_{n-d'}-N^{p,p'}_{n-d'}$$
Proceeding in the same was as the proof of Proposition \ref{mean}, we get
\begin{align*}
N^{p,p'}_n = \sum_{k=1}^{\lfloor \frac{n}{d'} \rfloor} (-1)^{k+1}N^p_{n-kd'} + O(q^d). 
\end{align*}
Proposition \ref{mean} give us the size of the $P^p_{n-kd'}$, so 
\begin{align*}
N^{p,p'}_n = \sum_{k=1}^{\lfloor \frac{n}{d'} \rfloor} (-1)^{k+1} \#P_{n-kd'} \frac{1}{q^d+1} + O(q^d). 
\end{align*}
Letting $n \to \infty$ we evaluate the geometric series and find 
\begin{equation} \label{ppn}
N^{p,p'}_n = (q^n-q^{n-1})\frac{1}{q^d+1}\frac{1}{q^{d'}+1}+ O(q^d).
\end{equation}
Substituting \eqref{ppn} back into \eqref{EsqX} and rewriting, we get
\begin{equation}\label{e(x2)notsplit}
\frac{1}{\#P_n}\sum_{d,d' \leq T}  \sum_{\substack{p,p' \text{ irred.} \\ \text{deg }p =d \\ \text{deg }p' = d'}} \sum_{\substack{f \in P_n \\ p,p' \mid f}} 1 = \sum_{d,d' \leq T}  \sum_{\substack{p,p' \text{ irred.} \\ \text{deg }p =d \\ \text{deg }p' = d'}} \frac{1}{q^d+1}\frac{1}{q^{d'}+1}+ \textcolor{black}{O(q^{d-n+1}).}
\end{equation}
We'll break the sum in \eqref{e(x2)notsplit} up into three cases depending on if $d=d'$ or not as follows:
\begin{equation}\label{e(x2)split1}
 \sum_{d\neq d' \leq T}  \sum_{\substack{p,p' \text{ irred.} \\ \text{deg }p =d \\ \text{deg }p' = d'}} \frac{1}{q^d+1}\frac{1}{q^{d'}+1} + \sum_{d = d' \leq T}  \sum_{\substack{p \neq p' \text{ irred.} \\ \text{deg }p =d \\ \text{deg }p' = d'}} \frac{1}{q^d+1}\frac{1}{q^{d'}+1}+ \sum_{d=d'\leq T}\sum_{\substack{p = p' \text{ irred.} \\ \text{deg }p =d}}\frac{1}{q^d+1}+\textcolor{black}{O(q^{d-n+1})}
\end{equation}
Notice in the above equation, the summands have no dependence on the irreducible polynomials  $p$ or $p'$ and so we these terms may be pulled out and counted as $\pi(d)$ or $\pi(d')$. 
Then, evaluating each of the three expressions, \eqref{e(x2)split1} is equal to 
\begin{equation} 
\sum_{d\neq d' \leq T} \pi(d)\pi(d') \frac{1}{q^d+1}\frac{1}{q^{d'}+1} + \sum_{d=d' \leq T} \left( \pi(d)^2 -\pi(d)\right)\left(\frac{1}{q^d+1} \right)^2 + \sum_{d=d'\leq T} \pi(d)\frac{1}{q^d+1}+\textcolor{black}{O(q^{T-n+1})}. \label{e(x^2)expanded}
\end{equation}
We'll now address the second term of (\ref{variance}).  We have, using Proposition \ref{mean}, that
\begin{align}
\notag \mu^2 &= \left(\frac{1}{\#P_n} \sum_{f \in P_n}  \sum_{\substack{p \text{ irred.} \\ p\mid f \\ \text{deg }p \leq T}} 1 \right)^2 \\
\notag &= \left(\sum_{d \leq T} \pi(d)\frac{1}{q^d+1} + O(q^{T-n+1}) \right)^2 \\
\notag &= \sum_{d,d' \leq T} \pi(d)\pi(d')\frac{1}{q^d+1}\frac{1}{q^{d'}+1} + O\left(q^{2(T-n+1)}\right) \\
&= \sum_{d = d'\leq T} \pi(d)^2 \left(\frac{1}{q^d+1} \right)^2 + \sum_{d \neq d' \leq T} \pi(d) \pi(d')\frac{1}{q^d+1}\frac{1}{q^{d'}+1} + O\left(q^{2(T-n+1)}\right) \label{e(x)^2expanded}
\end{align}

Finally, taking the difference of \eqref{e(x^2)expanded} and \eqref{e(x)^2expanded} to get $\sigma^2$ as in \eqref{variance}, we obtain
$$\sigma^2  =  \sum_{d} \pi(d)\frac{1}{q^d+1}\left(1-\frac{1}{q^d+1}\right) + O\left(q^{2(T-n+1)}\right)$$
as desired. 
\end{proof} 

\begin{cor}\label{meanexp} With the notation as above, and as $n \to \infty$, we have 
$$\mu \sim \log T$$
and 
$$\sigma^2 \sim \log T.$$
\end{cor}

\begin{proof}
Starting with the conclusion of Proposition \ref{mean}, we have as $n \to \infty$ that
\begin{equation} \label{mu}
\mu \sim \sum_{d \leq T} \pi(d)\frac{1}{q^d + 1}.
\end{equation}

It is known, see e.g. \cite{Rosen}, that 
\begin{equation}\label{pnt}
\pi(d) = \frac{q^d}{d}+O\left(\frac{q^{d/2}}{d}\right)
\end{equation}
From (\ref{mu}) and (\ref{pnt}) one obtains 
\begin{equation} \label{tog}
\sum_{d \leq T} \pi(d)\frac{1}{q^d + 1} = \sum_{d < T} \left( \frac{q^d}{d(q^d+1)}+O(q^{-d/2})\right) = \sum_{d<T} \frac{q^d}{d(q^d+1)} + O\left(q^{-1/2}\right)
\end{equation}
Now we compute the sum on the right side of \eqref{tog}, 

\begin{align}
\notag \sum_{d<T} \frac{q^d}{d(q^d+1)}  &= \sum_{d<T}\frac{1}{d} - \sum_{d<T}\frac{1}{d}\left(\frac{1}{q^d+1}\right) \\
\notag &= \log T + O(1/T) - \sum_{d<T}\frac{1}{d}\left(\frac{1}{q^d+1}\right) \\
& = \log T + O\left( \frac{1}{T} \right)
\end{align}
Where the final equality comes by integrating the last sum by parts and trivially bounding the integrand by $1/q^t$. The result follows. The analysis for $\sigma^2$ is essentially the same.
\end{proof}

\begin{theorem}\label{irredfacts}
As $n \to \infty$,  all but a proportion of $\frac{1}{\log{\frac{n}{2}}}$ square-free polynomials $f \in P_n \subset \mathbb{F}_q[x]$ are such that $f$ has at least $\log \frac{n}{2} + O\left(\sqrt{\log \frac n 2}\right)$ irreducible factors. 
\end{theorem}

\begin{proof}
Set $k= \sqrt{\log \frac n 2}$. Now apply Theorem \ref{Cheby} with $X=P_n$, $\phi=\omega_T$, and with $\mu$ and $\sigma^2$ given by Corollary \ref{meanexp} setting $T=\frac n 2$. One finds, 
$$\frac{\#\{f\in P_n \mid ~ |\omega_{\frac n 2}(f) - \log \frac n 2| \geq \log \frac n 2\}}{\#P_n} \leq \frac{1}{\log \frac n 2}.$$
Consequently, 
$$\lim_{n \to \infty}\frac{\#\{f\in P_n \mid ~ |\omega_{\frac n 2}(f) - \log \frac n 2| \leq \log \frac n 2\}}{q^n - q^{n-1}} \geq \lim_{n \to \infty} 1 - \frac{1}{\log \frac n 2} = 1. $$
\end{proof}

\begin{proof}[Proof of Theorem \ref{100per}]
We apply Theorem \ref{Cheby} with  $X=P_d$, $\phi=\omega_T$, $T=\frac{d}{2}$, and $k = \alpha \sqrt{\log{\frac{d}{2}}}$ where $\alpha = \beta + O\left(\frac{1}{d \sqrt{ \log \frac d 2}}\right)$.

For the parameter $\alpha$, we'll make a choice of $\beta$ and a choice of error term in the course of the proof. We essentially mimic the proof of Theorem \ref{irredfacts} and then apply Theorem \ref{MKff}.

With $\mu$ and $\sigma$ as above, an immediate consequence of Theorem \ref{Cheby} is that for a random $f(x) \in P_d$, 
\begin{equation}\label{chebyin}
    \omega_{d/2}(f) \geq \mu - k\sigma 
\end{equation}
with probability at least $1-\frac{1}{k^2}$. If \eqref{chebyin} holds, Corollary \ref{meanexp} implies 
\begin{equation}\label{chebyin2}
    \omega_{d/2}(f) \geq (1-\alpha)\log\frac{d}{2} + O\left( \frac{\sqrt{\log \frac{d}{2}}}{d}\right).
\end{equation}
Pick $\alpha$ as follows: pick the error term which is the negative of the implicit error appearing in \eqref{chebyin2} and pick any $0 < \beta < 1$. Note that the error terms we're countering with this choice, coming from Corollary \ref{meanexp}, depend only on $d$, not on $f(x)$. 

For convenience, pick $\beta = 1/2$. We have then that a proportion of least $1-\frac{1}{k^2}$ of $f(x) \in P_d$ have at least $\beta \log \frac n 2$ irreducible factors. 

Let $F$ be the quadratic extensions of $\mathbb{F}_q(x)$ corresponding to a hyperelliptic curve $y^2=f(x)$ where $f(x) \in P_d \subset \mathbb{F}_q[x]$. Note $d = \deg f(x)=2g+1$ or $2g+2$ where $g$ is the genus of $F$. The discussion above shows that as $f(x)$ ranges through $P_d$, a proportion of  $1-\frac{1}{k^2}=1-O\left(\frac{1}{\log g}\right)$ of the associated $F$, are such that $\text{Cl}_F[2] = 2^{\omega(f)} > 2^{\beta \log g}$ where $\omega(f)$ is the number of irreducible factors of $f(x)$. 

So, for a proportion of $1-O\left(\frac{1}{\log g}\right)$ hyperelliptic genus $g$ extensions $F/\mathbb{F}_q(x)$, we have
\begin{align*}
    \lim_{n \to \infty}\frac{N_4^F(q^{2n};D_4)}{N_4^F(q^{2n};S_4)} &\gg \#\rm{Cl}_F[2]\left(1-\frac{1}{\sqrt{q}}\right)^{4g-2} \\
    &=2^{\omega_T(f)} \left(1-\frac{1}{\sqrt{q}}\right)^{4g-2} \\
    &\geq 2^{\beta \log g} \left(1-\frac{1}{\sqrt{q}}\right)^{4g-2} \\
    &= g^{\beta \log 2}\left(1-\frac{1}{\sqrt{q}}\right)^{4g-2},
\end{align*}
proving the theorem.
\end{proof}
For completeness, note that in the proof above any choice $0 < \beta < 1$ will do. One can thus slightly improve the exponent of $g$ stated in Theorem \ref{100per}. 

Finally, note that taking $q \to \infty$ and large $g$ in Theorem \ref{100per} is essentially an extremal version of \cite[Corollary 1.2]{FK} but for function fields rather than number fields. It is plausible that the $L$-function techniques used in the number field version could be ported over to the function field setting in order to ease the condition on $q$. 

\section{Acknowledgements}
 Many thanks to Robert Lemke Oliver for numerous helpful discussions and suggestions. Thanks also to Shamil Asgarli, Matt Friedrichsen, and George McNinch for helpful conversations.

\end{document}